\documentclass{amsart}

\makeatletter
\@namedef{subjclassname@2020}{%
  \textup{2020} Mathematics Subject Classification}
\makeatother

\usepackage{amsthm,amssymb,amsfonts,amsmath,mathtools}
\usepackage{mathrsfs}
\usepackage[T1]{fontenc}
\usepackage{longtable}
\usepackage{array}
\usepackage{booktabs}
\usepackage{hyperref}

\hypersetup{
    colorlinks=true,
    linkcolor=blue,
    filecolor=blue,
    urlcolor=blue,
    linktocpage=true
}

\newtheorem{theorem}{Theorem}[section]

\newtheorem{proposition}[theorem]{Proposition}

\theoremstyle{definition}
\newtheorem{definition}[theorem]{Definition}
\newtheorem{corollary}[theorem]{Corollary}
\newtheorem{remark}[theorem]{Remark}

\theoremstyle{remark}

\numberwithin{equation}{section}
\allowdisplaybreaks

\begin{document}

\title[Normal ordering for double Ore extensions]{Normal Ordering and Stirling-Type Combinatorics for Double Ore Extensions of Type $(14641)$}

%    author one information

\author{Andr\'es Rubiano}
\address{Universidad Distrital Francisco Jos\'e de Caldas}
\curraddr{Campus Universitario}
\email{aarubianos@udistrital.edu.co}
\address{Universidad ECCI}
\curraddr{Campus Universitario}
\email{arubianos@ecci.edu.co}
\thanks{}



\date{}

\dedicatory{Dedicated to Adela Le\'on}

\begin{abstract}
We develop an explicit PBW normal ordering theory for the $26$ double extension regular algebras of type $(14641)$ in the Zhang--Zhang classification. With respect to the order $x_1\prec x_2\prec y_1\prec y_2$, we obtain closed two-letter formulas for the internal relations and recursive coefficient systems for mixed words, products of PBW monomials, powers of normal blocks, and noncommutative multinomial expressions. The internal coefficients are mostly quantum or skew-commutative, while the Jordan families produce Lah--Whitney, hence Stirling-type, triangular arrays. The symbolic reductions are supported by a \textsc{SageMath} implementation included as an ancillary file.
\end{abstract}

\maketitle

%\tableofcontents 

\section{Introduction}

Normal ordering is one of the most efficient ways in which noncommutative algebra produces explicit combinatorics. In its classical form, the problem goes back to the Weyl relation $DX=XD+1$: every word in the operators $X$ and $D$ can be rewritten uniquely as a linear combination of ordered monomials $X^iD^j$, and the coefficients obtained in this process include the Stirling numbers of the second kind. Since then, variants of the same phenomenon have appeared in quantum planes, $q$-Weyl algebras, creation and annihilation operator algebras, Ore extensions and other PBW-type algebras; see, for instance, \cite{Varvak2005,MansourSchork2015,Schork2021,BriandLopesRosas2020,MansourSchork2023}.

The common mechanism behind these examples is a rewriting system. One fixes an admissible order on the generators and repeatedly eliminates the forbidden adjacent pairs until an ordered expression is obtained. When the defining relations give a PBW basis, the normal form is unique; equivalently, the reduction process produces well-defined normal ordering coefficients. Bergman's Diamond Lemma \cite{Bergman1978} provides the conceptual background for this viewpoint, while concrete PBW computations turn the abstract confluence principle into explicit recurrences and coefficient arrays.

This rewriting philosophy was recently used by Reyes and the author in the normal ordering study of the $3$-dimensional skew polynomial rings classified by Bell and Smith \cite{BellSmith1990,RubianoReyes2026Combinatorial3D}. The present paper should be viewed as a four-generator continuation of that program: the PBW basis is now indexed by four exponents, and the reduction data contain two internal systems together with four mixed crossing rules.

Ore extensions, introduced by Ore in \cite{Ore1931,Ore1933}, form one of the fundamental sources of such examples. If $R$ is a ring, $\sigma$ an endomorphism of $R$, and $\delta$ a $\sigma$-derivation, the Ore extension $R[x;\sigma,\delta]$ is generated by $R$ and an element $x$ subject to $xr=\sigma(r)x+\delta(r)$. Normal ordering in Ore-type algebras is already rich: even when there are only two generators, changing the commutation rule changes the associated coefficient arrays from ordinary Stirling numbers to $q$-analogues, rook numbers, Whitney-type numbers and other generalized triangular families.

The purpose of this paper is to develop the corresponding normal ordering theory for a substantially larger class: double extension regular algebras of type $(14641)$. Double Ore extensions were introduced by Zhang and Zhang \cite{ZhangZhang2008,ZhangZhang2009} as a two-generator analogue of Ore extensions. In the trimmed case, a double extension over a two-generator Artin--Schelter regular algebra has generators $x_1,x_2,y_1,y_2$ and defining relations of the form
\[
x_2x_1=q_{12}x_1x_2+q_{11}x_1^2,
\
 y_2y_1=p_{12}y_1y_2+p_{11}y_1^2,
\]
together with four mixed relations
\[
y_ix_s=\sum_{j=1}^2\sum_{t=1}^2 a_{ijst}x_ty_j,
\ i,s\in\{1,2\}.
\]
Thus, after fixing the order
\[
x_1\prec x_2\prec y_1\prec y_2,
\]
the forbidden adjacent pairs are $x_2x_1$, $y_2y_1$ and $y_ix_s$. This is the first point at which the double extension setting becomes genuinely different from the usual two-generator normal ordering problems: there are two internal systems and four mixed crossing rules, so the coefficient arrays must record several interacting layers of reductions.

The families considered here are the $26$ double extension regular algebras of type $(14641)$ classified by Zhang and Zhang \cite{ZhangZhang2009}. These algebras are Artin--Schelter regular algebras of global dimension four and are related to the general classification program for noncommutative analogues of polynomial algebras; see \cite{ArtinSchelter1987,Rogalski2023}. The compatibility between the double extension data and iterated Ore presentations was clarified by Carvalho, Lopes and Matczuk \cite{Carvalhoetal2011}. In the present work we do not reprove the classification. Instead, we take the explicit Zhang--Zhang families as input and study their algebraic combinatorics.

Our first result is a closed formula for the two internal reductions. If
\[
x_2x_1=qx_1x_2+rx_1^2,
\]
then
\[
x_2^nx_1^m
=
\sum_{k=0}^{n}
 r^kq^{m(n-k)}
\begin{bmatrix} n\\ k \end{bmatrix}_{q}
[m]_{q}^{\overline{k}}
 x_1^{m+k}x_2^{n-k}.
\]
The same formula holds for the $y$-variables after replacing $(q,r)$ by the corresponding parameters $(p,s)$. This already separates the $26$ families into two regimes. Most internal systems are quantum or skew-commutative and therefore produce scalar factors. By contrast, the Jordan cases $\mathbb{A}$ and $\mathbb{H}$ produce the triangular coefficients
\[
\binom{n}{k}m^{\overline{k}},
\]
which are Lah--Whitney coefficients and may be viewed as homogeneous relatives of the Stirling numbers arising in the Weyl algebra.

The second part of the paper develops explicit mixed normal ordering. For each family $\bar{\mathbb{F}}$ we introduce crossing kernels
\[
\mathscr{K}_{i,s}^{\bar{\mathbb{F}}}
\]
which list all summands in the normal form of $y_ix_s$. These kernels are given explicitly for all $26$ families. Using them, we define coefficient arrays
\[
W_{\bar{\mathbb{F}},i}^{(a,b)}(\alpha,\beta;\ell)
\]
for the normal ordering of $y_ix_1^ax_2^b$, and arrays
\[
\Theta_{\bar{\mathbb{F}}}(a,b;c,d;\alpha,\beta;\gamma,\delta)
\]
for the normal ordering of $y_1^cy_2^dx_1^ax_2^b$. These arrays are then used to obtain explicit structure constants for products of PBW monomials and recursive formulas for powers of normal blocks
\[
(x_1^ax_2^by_1^cy_2^d)^N
\]
and for powers of arbitrary linear combinations
\[
\lambda_1x_1+\lambda_2x_2+\mu_1y_1+\mu_2y_2.
\]
In this way, the paper gives a complete recursive normal ordering solution for every family in $\mathcal{LIST}$.

Several computations were verified using \textsc{SageMath} \cite{SageMath2026}. The implementation represents noncommutative polynomials as finite dictionaries of words and coefficients, applies the PBW rewriting rules, and computes normal forms, commutators, products, powers and centrality systems in fixed bidegrees. Since the code is too long to include in the body of the article, it is provided as the ancillary file \texttt{double\_ore\_pbw.sage} with the arXiv version of this manuscript.

The article is organized as follows. Section \ref{Preliminaries} recalls the necessary background on double extensions, the type $(14641)$ families, PBW bases and normal ordering. Section \ref{SectionNormalOrderingDOE} contains the main normal ordering results: internal formulas, explicit crossing kernels for the $26$ families, mixed coefficient recursions, structure constants, block powers and noncommutative multinomial recurrences. Section \ref{SubsectionComputationalImplementation} explains the computational implementation. Section \ref{Futurework} discusses future directions.

Throughout the paper, $\bar{\mathbb{N}}$ denotes the set of natural numbers including zero. The word \emph{ring} means an associative ring with identity, not necessarily commutative. All vector spaces and algebras are over a fixed field $\Bbbk$, which is assumed to be algebraically closed and of characteristic zero. Unless otherwise stated, all algebras are associative and unital.
\section{Preliminaries}\label{Preliminaries}

In this section we fix the notation and terminology used throughout the paper. We first recall the basic language of double Ore extensions, then specialize to double extensions over Artin--Schelter regular algebras of global dimension two, and finally introduce the normal ordering convention that will be used in the combinatorial computations.

\subsection{Double Ore extensions}\label{SubsectionDoubleOreExtensions}

We begin by recalling the notion of double extension introduced by Zhang and Zhang \cite{ZhangZhang2008}. We use the corrected formulation of the compatibility conditions given by Carvalho et al. \cite{Carvalhoetal2011}.

\begin{definition}[{\cite[Definition 1.3]{ZhangZhang2008}; \cite[Definition 1.1]{Carvalhoetal2011}}]\label{DefinitionDoubleExtension}
Let $R$ be a subalgebra of a $\Bbbk$-algebra $B$.
\begin{enumerate}
    \item[\rm (a)] We say that $B$ is a \emph{right double extension} of $R$ if the following conditions hold:
    \begin{itemize}
        \item[\rm (i)] $B$ is generated by $R$ and two elements $y_1,y_2$;
        \item[\rm (ii)] the generators $y_1,y_2$ satisfy a quadratic relation of the form
        \begin{equation}\label{RelationY2Y1General}
        y_2y_1
        =
        p_{12}y_1y_2+p_{11}y_1^2+\tau_1y_1+\tau_2y_2+\tau_0,
        \end{equation}
        where $p_{12},p_{11}\in\Bbbk$ and $\tau_0,\tau_1,\tau_2\in R$;
        \item[\rm (iii)] $B$ is a free left $R$-module with basis
        \[
        \{y_1^iy_2^j\mid i,j\in\bar{\mathbb{N}}\};
        \]
        \item[\rm (iv)] one has
        \[
        y_1R+y_2R+R\subseteq Ry_1+Ry_2+R.
        \]
    \end{itemize}

    \item[\rm (b)] A right double extension $B$ of $R$ is called a \emph{double extension} of $R$ if, in addition,
    \begin{itemize}
        \item[\rm (i)] $p_{12}\neq 0$;
        \item[\rm (ii)] $B$ is a free right $R$-module with basis
        \[
        \{y_2^iy_1^j\mid i,j\in\bar{\mathbb{N}}\};
        \]
        \item[\rm (iii)] one has
        \[
        y_1R+y_2R+R=Ry_1+Ry_2+R.
        \]
    \end{itemize}
\end{enumerate}
\end{definition}

Condition \ref{DefinitionDoubleExtension}(a)(iv) is equivalent to the existence of maps
\[
\sigma=
\begin{bmatrix}
\sigma_{11} & \sigma_{12}\\
\sigma_{21} & \sigma_{22}
\end{bmatrix}:R\to M_{2\times 2}(R),
\
\delta=
\begin{bmatrix}
\delta_1\\
\delta_2
\end{bmatrix}:R\to M_{2\times 1}(R),
\]
such that
\begin{equation}\label{MatrixRelationDoubleExtension}
\begin{bmatrix}
y_1\\
y_2
\end{bmatrix}r
=
\sigma(r)
\begin{bmatrix}
y_1\\
y_2
\end{bmatrix}
+
\delta(r),
\ r\in R.
\end{equation}
When this happens, we write
\[
B=R_P[y_1,y_2;\sigma,\delta,\tau],
\]
where $P=(p_{12},p_{11})\in\Bbbk^2$ and $\tau=\{\tau_0,\tau_1,\tau_2\}\subseteq R$. The collection $\{P,\sigma,\delta,\tau\}$ will be called the \emph{DE-data} of $B$.

The special case in which $\delta=0$ and $\tau=\{0,0,0\}$ is particularly important for us.

\begin{definition}\label{DefinitionTrimmedDoubleExtension}
A right double extension $R_P[y_1,y_2;\sigma,\delta,\tau]$ is called \emph{trimmed} if $\delta=0$ and $\tau=\{0,0,0\}$. In this case we write simply
\[
R_P[y_1,y_2;\sigma].
\]
Thus a trimmed double extension is generated by $R,y_1,y_2$ subject to
\[
y_2y_1=p_{12}y_1y_2+p_{11}y_1^2
\]
and
\[
\begin{bmatrix}
y_1\\
y_2
\end{bmatrix}r
=
\sigma(r)
\begin{bmatrix}
y_1\\
y_2
\end{bmatrix},
\ r\in R.
\]
\end{definition}

For a right double extension $R_P[y_1,y_2;\sigma,\delta,\tau]$, the map $\sigma:R\to M_{2\times 2}(R)$ is an algebra homomorphism and $\delta$ is a $\sigma$-derivation, that is,
\[
\delta(rr')=\sigma(r)\delta(r')+\delta(r)r',
\ r,r'\in R.
\]
In particular, in the trimmed case the interaction between $R$ and the new variables $y_1,y_2$ is completely governed by the entries of the matrix $\sigma$.

\begin{remark}\label{RemarkDoubleVsIteratedOre}
Double extensions should not be confused with two-step iterated Ore extensions. Although many double extensions can be written as iterated Ore extensions, this is not automatic. The obstruction comes from the fact that an iterated Ore extension need not admit a quadratic relation of the form \eqref{RelationY2Y1General}. Conversely, a double extension may fail to be an iterated Ore extension. Criteria deciding when a double extension admits an iterated Ore presentation were obtained in \cite{Carvalhoetal2011}.
\end{remark}

\subsection{Regular algebras of type \texorpdfstring{$(14641)$}{(14641)}}\label{SubsectionType14641}

We recall the class of double extension regular algebras that will be considered in this paper. Let $B$ be a connected graded algebra generated in degree one. If the trivial right module $\Bbbk_B$ has a minimal projective resolution of the form
\begin{equation}\label{Resolution14641}
0
\longrightarrow B(-4)
\longrightarrow B(-3)^{\oplus 4}
\longrightarrow B(-2)^{\oplus 6}
\longrightarrow B(-1)^{\oplus 4}
\longrightarrow B
\longrightarrow \Bbbk_B
\longrightarrow 0,
\end{equation}
then $B$ is said to be \emph{of type} $(14641)$.

Zhang and Zhang \cite{ZhangZhang2009} proved that, under suitable hypotheses, double extensions over Artin--Schelter regular algebras of global dimension two produce Artin--Schelter regular algebras of global dimension four of type $(14641)$. More precisely, if $R$ is an Artin--Schelter regular algebra of global dimension two and
\[
B=R_P[y_1,y_2;\sigma]
\]
is a connected graded double extension generated in degree one, then $B$ is strongly Noetherian, Auslander regular, Cohen--Macaulay, Koszul, and of type $(14641)$; see \cite[Theorem 0.1]{ZhangZhang2009}.

The classification in \cite{ZhangZhang2009} gives $26$ families of such trimmed double extensions, which are labeled
\[
\mathbb{A},\mathbb{B},\dots,\mathbb{Z}.
\]
Following the notation used in the literature, we denote by $\mathcal{LIST}$ the class of all algebras belonging to these families.

The families in $\mathcal{LIST}$ have been studied from several complementary viewpoints. Reyes and the author proved that double extension regular algebras of type $(14641)$ are not differentially smooth in the sense of noncommutative differential geometry \cite{RubianoReyes2024Smooth}. Herrera, Higuera and the author computed finite Gr\"obner--Shirshov bases for several such families and showed that they admit PBW bases by an algorithmic method \cite{HerreraHigueraRubiano2025GS}. The center and certain central subalgebras of selected families were computed by the author, with applications to cancellation phenomena \cite{Rubiano2026Center}. Explicit homological invariants of graded double Ore extensions have also been investigated in \cite{Rubiano2026HomologicalInvariants}. The present paper adds a combinatorial layer to this line of work by extracting normal ordering recurrences and coefficient arrays from the same PBW structure.

\subsection{Double extensions over \texorpdfstring{$\Bbbk_Q[x_1,x_2]$}{kQ[x1,x2]}}\label{SubsectionDoubleExtensionsKQ}

Let $Q=(q_{12},q_{11})\in\Bbbk^2$ with $q_{12}\neq 0$. We write
\begin{equation}\label{DefinitionKQ}
\Bbbk_Q[x_1,x_2]
:=
\Bbbk\{x_1,x_2\}\big/
\left\langle
x_2x_1-q_{12}x_1x_2-q_{11}x_1^2
\right\rangle.
\end{equation}
Thus, in $\Bbbk_Q[x_1,x_2]$ one has
\begin{equation}\label{RelationX2X1}
x_2x_1=q_{12}x_1x_2+q_{11}x_1^2.
\end{equation}

The Artin--Schelter regular algebras of global dimension two are, up to isomorphism, the Manin plane and the Jordan plane \cite[Theorem 1.4]{Shirikov2005}. In the present notation, these correspond respectively to
\[
Q=(q,0),\ q\in\Bbbk^\times,
\]
and
\[
Q=(1,1).
\]
Hence the notation $\Bbbk_Q[x_1,x_2]$ allows us to treat both cases simultaneously.

Let
\[
B=(\Bbbk_Q[x_1,x_2])_P[y_1,y_2;\sigma]
\]
be a trimmed double extension, where $P=(p_{12},p_{11})$ with $p_{12}\neq 0$. Then $B$ is generated by $x_1,x_2,y_1,y_2$ and has two non-mixing quadratic relations:
\begin{align}
x_2x_1 &= q_{12}x_1x_2+q_{11}x_1^2, \label{NonMixingXRelation}\\
y_2y_1 &= p_{12}y_1y_2+p_{11}y_1^2. \label{NonMixingYRelation}
\end{align}

The mixed relations are encoded by the algebra homomorphism
\[
\sigma:\Bbbk_Q[x_1,x_2]\to M_{2\times 2}(\Bbbk_Q[x_1,x_2]).
\]
Since we work in the connected graded case, $\sigma$ is determined by its values on $x_1$ and $x_2$. We write
\begin{equation}\label{SigmaCoordinates}
\sigma_{ij}(x_s)=\sum_{t=1}^2 a_{ijst}x_t,
\ i,j,s\in\{1,2\},
\end{equation}
where $a_{ijst}\in\Bbbk$. Thus, for $i,s\in\{1,2\}$,
\begin{equation}\label{GeneralMixedRelation}
y_ix_s
=
\sum_{j=1}^2\sum_{t=1}^2 a_{ijst}x_ty_j.
\end{equation}
Equivalently, the four mixed relations are
\begin{align}
y_1x_1
&=
a_{1111}x_1y_1+a_{1112}x_2y_1
+a_{1211}x_1y_2+a_{1212}x_2y_2, \label{MixedRelation11}\\
y_1x_2
&=
a_{1121}x_1y_1+a_{1122}x_2y_1
+a_{1221}x_1y_2+a_{1222}x_2y_2, \label{MixedRelation12}\\
y_2x_1
&=
a_{2111}x_1y_1+a_{2112}x_2y_1
+a_{2211}x_1y_2+a_{2212}x_2y_2, \label{MixedRelation21}\\
y_2x_2
&=
a_{2121}x_1y_1+a_{2122}x_2y_1
+a_{2221}x_1y_2+a_{2222}x_2y_2. \label{MixedRelation22}
\end{align}

For later reference, we also introduce the matrices
\begin{equation}\label{SigmaMatrices}
\Sigma_{ij}:=
\begin{bmatrix}
a_{ij11} & a_{ij12}\\
a_{ij21} & a_{ij22}
\end{bmatrix},
\qquad
\Sigma:=
\begin{bmatrix}
\Sigma_{11} & \Sigma_{12}\\
\Sigma_{21} & \Sigma_{22}
\end{bmatrix}.
\end{equation}
Explicitly,
\[
\Sigma=
\begin{bmatrix}
a_{1111} & a_{1112} & a_{1211} & a_{1212}\\
a_{1121} & a_{1122} & a_{1221} & a_{1222}\\
a_{2111} & a_{2112} & a_{2211} & a_{2212}\\
a_{2121} & a_{2122} & a_{2221} & a_{2222}
\end{bmatrix}.
\]
The entries of $\Sigma$ are subject to the compatibility equations ensuring that \eqref{NonMixingXRelation}, \eqref{NonMixingYRelation}, and \eqref{MixedRelation11}--\eqref{MixedRelation22} define a double extension. Following \cite{ZhangZhang2009}, such a matrix will be called a $C$-\emph{solution} whenever it satisfies those compatibility equations and ${\rm det}\,\Sigma\neq 0$.

\begin{remark}\label{RemarkCSolutions}
In the computations below we do not need the full parameter-dependent form of the compatibility system. Instead, for each algebra in $\mathcal{LIST}$ we use the explicit relations defining that family. The role of the $C$-solution condition is to guarantee that the displayed relations determine a double extension and, consequently, that the PBW normal forms described below are well defined.
\end{remark}

\subsection{PBW bases and normal ordering}\label{SubsectionPBWNormalOrdering}

The combinatorial problem studied in this paper is the problem of expressing arbitrary words in the generators $x_1,x_2,y_1,y_2$ in PBW normal form.

We fix throughout the monomial order
\begin{equation}\label{PBWOrder}
x_1\prec x_2\prec y_1\prec y_2.
\end{equation}
With this order, a monomial is said to be \emph{normal ordered} if it has the form
\[
x_1^ax_2^by_1^cy_2^d,
\ a,b,c,d\in\bar{\mathbb{N}}.
\]

\begin{proposition}\label{PropositionPBWBasis}
Let
\[
B=(\Bbbk_Q[x_1,x_2])_P[y_1,y_2;\sigma]
\]
be a trimmed double extension. Then the set
\[
\mathcal{B}
:=
\left\{
x_1^ax_2^by_1^cy_2^d
\mid
a,b,c,d\in\bar{\mathbb{N}}
\right\}
\]
is a $\Bbbk$-basis of $B$.
\end{proposition}

\begin{proof}
By definition, $\Bbbk_Q[x_1,x_2]$ has PBW basis
\[
\{x_1^ax_2^b\mid a,b\in\bar{\mathbb{N}}\}.
\]
Since $B$ is a right double extension of $\Bbbk_Q[x_1,x_2]$, it is a free left $\Bbbk_Q[x_1,x_2]$-module with basis
\[
\{y_1^cy_2^d\mid c,d\in\bar{\mathbb{N}}\}.
\]
Therefore every element of $B$ can be written uniquely as a finite $\Bbbk$-linear combination of monomials
\[
x_1^ax_2^by_1^cy_2^d,
\]
and the result follows.
\end{proof}

The normal ordering procedure is the rewriting process that eliminates all adjacent pairs that are not compatible with \eqref{PBWOrder}. Thus the forbidden adjacent pairs are
\[
x_2x_1,\ y_2y_1,\ y_ix_s
\text{ for } i,s\in\{1,2\}.
\]
They are rewritten using \eqref{NonMixingXRelation}, \eqref{NonMixingYRelation}, and \eqref{GeneralMixedRelation}, respectively:
\begin{align*}
x_2x_1
&\longmapsto
q_{12}x_1x_2+q_{11}x_1^2,\\
y_2y_1
&\longmapsto
p_{12}y_1y_2+p_{11}y_1^2,\\
y_ix_s
&\longmapsto
\sum_{j=1}^2\sum_{t=1}^2 a_{ijst}x_ty_j.
\end{align*}

\begin{definition}\label{DefinitionNormalForm}
Let $\omega$ be a word in the alphabet $\{x_1,x_2,y_1,y_2\}$. The \emph{normal ordered form} of $\omega$, denoted by ${\rm NF}(\omega)$, is its unique expansion in the PBW basis $\mathcal{B}$:
\begin{equation}\label{NormalFormExpansion}
{\rm NF}(\omega)
=
\sum_{a,b,c,d\in\bar{\mathbb{N}}}
C_{a,b,c,d}(\omega)\,
x_1^ax_2^by_1^cy_2^d,
\end{equation}
where only finitely many coefficients $C_{a,b,c,d}(\omega)\in\Bbbk$ are nonzero. The scalars $C_{a,b,c,d}(\omega)$ are called the \emph{normal ordering coefficients} of $\omega$.
\end{definition}

More generally, if $f$ is a finite $\Bbbk$-linear combination of words, then ${\rm NF}(f)$ is defined by extending \eqref{NormalFormExpansion} linearly.

\begin{remark}\label{RemarkNormalOrdering}
The existence and uniqueness of ${\rm NF}(\omega)$ follow from Proposition \ref{PropositionPBWBasis}. In practice, ${\rm NF}(\omega)$ is obtained by repeatedly replacing forbidden adjacent pairs by the corresponding right-hand sides above. Each replacement moves the word closer to the order $x_1\prec x_2\prec y_1\prec y_2$, and the PBW property guarantees that the final result is independent of the chosen sequence of reductions.
\end{remark}

The basic objects of study in the next sections will be normal ordering identities for words such as
\[
x_2^nx_1^m,\ y_2^ny_1^m,\ y_ix_1^ax_2^b,
\
\left(x_1^ax_2^by_1^cy_2^d\right)^N,
\]
and for noncommutative multinomial-type expressions. These identities will be expressed through explicit coefficient families satisfying recursive relations.

\subsection{Combinatorial notation}\label{SubsectionCombinatorialNotation}

We collect here the elementary notation used in the normal ordering formulas. We use the convention that empty sums are equal to $0$ and empty products are equal to $1$. The Kronecker delta is denoted by $\delta_{i,j}$. Unless otherwise stated, every coefficient family indexed outside its natural range is understood to be zero.

For $n,k\in\bar{\mathbb{N}}$, the usual binomial coefficient is denoted by
\[
\binom{n}{k}
=
\frac{n!}{k!(n-k)!}
\]
for $0\le k\le n$, and we set $\binom{n}{k}=0$ if $k>n$.

\begin{definition}\label{DefinitionQIntegers}
Let $q\in\Bbbk^\times$ and $n\in\bar{\mathbb{N}}$. The $q$-\emph{integer} $[n]_q$ is defined by
\[
[n]_q:=1+q+\cdots+q^{n-1},
\ [0]_q:=0.
\]
The corresponding $q$-\emph{factorial} is
\[
[n]_q!:=\prod_{j=1}^n [j]_q,
\ [0]_q!:=1.
\]
The \emph{Gaussian binomial coefficient} is the polynomial in $q$ characterized by
\[
\begin{bmatrix}
0\\
0
\end{bmatrix}_q=1,\
\begin{bmatrix}
n\\
k
\end{bmatrix}_q=0
\text{ if }k<0\text{ or }k>n,
\]
and, for $n\ge0$,
\[
\begin{bmatrix}
n+1\\
k
\end{bmatrix}_q
=
\begin{bmatrix}
n\\
k-1
\end{bmatrix}_q
+
q^k
\begin{bmatrix}
n\\
k
\end{bmatrix}_q.
\]
Equivalently, whenever the quotient expression is defined, one has
\[
\begin{bmatrix}
n\\
k
\end{bmatrix}_q
=
\frac{[n]_q!}{[n-k]_q![k]_q!}.
\]
\end{definition}

\begin{definition}\label{DefinitionQRisingFactorial}
Let $q\in\Bbbk^\times$ and $m,k\in\bar{\mathbb{N}}$. The $q$-\emph{rising factorial} is
\[
[m]_{q}^{\overline{k}}
:=
\prod_{j=0}^{k-1}[m+j]_q,
\
[m]_{q}^{\overline{0}}:=1.
\]
When $q=1$, this becomes the usual rising factorial
\[
[m]_1^{\overline{k}}
=
m(m+1)\cdots(m+k-1).
\]
\end{definition}

\begin{definition}\label{DefinitionLambdaCoefficients}
Let $q\in\Bbbk^\times$ and $r\in\Bbbk$. For $n,m,k\in\bar{\mathbb{N}}$, define
\[
\Lambda_{q,r}(n,m;k)
:=
r^kq^{m(n-k)}
\begin{bmatrix}
n\\
k
\end{bmatrix}_{q}
[m]_{q}^{\overline{k}},
\]
with the convention that $\Lambda_{q,r}(n,m;k)=0$ whenever $k>n$.
\end{definition}

\begin{remark}\label{RemarkCoefficientFamilies}
The coefficient families introduced below are local to the corresponding PBW expansions. The symbols
\[
W_{\bar{\mathbb{F}},i}^{(a,b)},\
\Theta_{\bar{\mathbb{F}}},\
\mathfrak{m}_{\bar{\mathbb{F}}},\
R_{\bar{\mathbb{F}},N}^{(a,b,c,d)},\
E_{\bar{\mathbb{F}},N}^{L}
\]
will respectively encode one-letter mixed reductions, mixed block reductions, products of PBW monomials, powers of normal blocks, and powers of noncommutative linear forms. Each family is defined together with its initial values and recursion.
\end{remark}

\section{Normal ordering identities for double Ore extensions of type \texorpdfstring{$(14641)$}{(14641)}}\label{SectionNormalOrderingDOE}

In this section we develop the normal ordering machinery for the $26$ families of double Ore extensions of type $(14641)$ in the Zhang--Zhang classification. The guiding principle is the following: once the PBW order
\[
x_1\prec x_2\prec y_1\prec y_2
\]
is fixed, every normal ordering problem is reduced to eliminating the forbidden adjacent pairs
\[
x_2x_1,\ y_2y_1,\ y_ix_s \text{ for } i,s\in\{1,2\}.
\]
Thus the computations below are governed by two internal two-letter systems, namely the $x$-system and the $y$-system, together with the four mixed relations. The internal systems already produce nontrivial combinatorial arrays, including $q$-binomial coefficients and, in the Jordan cases, Lah--Whitney type coefficients which play the role of homogeneous analogues of Stirling numbers in this setting.

With the notation fixed in Subsection \ref{SubsectionCombinatorialNotation}, the first result gives the closed normal ordering formulas for the two internal quadratic relations. This is the basic two-letter identity from which the rest of the section is built.

\begin{proposition}\label{PropositionInternalNormalOrdering}
Let $q,r,p,s\in\Bbbk$ with $q,p\in\Bbbk^\times$. Assume that
\[
x_2x_1=qx_1x_2+rx_1^2,
\qquad
y_2y_1=py_1y_2+sy_1^2.
\]
Then, for all $m,n\in\bar{\mathbb{N}}$, the following identities hold:
\begin{align}
x_2^nx_1^m
&=
\sum_{k=0}^{n}
\Lambda_{q,r}(n,m;k)\,
x_1^{m+k}x_2^{n-k}, \label{EquationXInternalNormalOrdering}\\
y_2^ny_1^m
&=
\sum_{k=0}^{n}
\Lambda_{p,s}(n,m;k)\,
y_1^{m+k}y_2^{n-k}. \label{EquationYInternalNormalOrdering}
\end{align}
\end{proposition}

\begin{proof}
We prove \eqref{EquationXInternalNormalOrdering}; the proof of \eqref{EquationYInternalNormalOrdering} is identical after replacing $(x_1,x_2,q,r)$ by $(y_1,y_2,p,s)$.

First observe that
\[
x_2x_1^m
=
q^mx_1^mx_2+r[m]_qx_1^{m+1}.
\]
This follows by induction on $m$. The case $m=1$ is the defining relation. If the identity holds for $m$, then
\begin{align*}
x_2x_1^{m+1}
&=
\left(q^mx_1^mx_2+r[m]_qx_1^{m+1}\right)x_1\\
&=
q^mx_1^m(qx_1x_2+rx_1^2)+r[m]_qx_1^{m+2}\\
&=
q^{m+1}x_1^{m+1}x_2+r(q^m+[m]_q)x_1^{m+2}\\
&=
q^{m+1}x_1^{m+1}x_2+r[m+1]_qx_1^{m+2}.
\end{align*}

Now write
\[
x_2^nx_1^m
=
\sum_{k=0}^{n}
C_{n,m}(k)x_1^{m+k}x_2^{n-k}.
\]
For $n=0$ one has $C_{0,m}(0)=1$. Multiplying the expansion for $n$ on the left by $x_2$ and using the previous identity gives
\[
C_{n+1,m}(k)
=
q^{m+k}C_{n,m}(k)
+
r[m+k-1]_qC_{n,m}(k-1),
\]
where $C_{n,m}(-1)=0$. A direct verification shows that
\[
C_{n,m}(k)
=
r^kq^{m(n-k)}
\begin{bmatrix}
n\\
k
\end{bmatrix}_q
[m]_{q}^{\overline{k}}
\]
satisfies this recurrence together with the initial condition.
\end{proof}

\begin{remark}\label{RemarkStirlingLahWhitney}
The coefficients in Proposition \ref{PropositionInternalNormalOrdering} are the first genuinely combinatorial coefficients of the double Ore setting. When $r=0$, the formula collapses to the familiar quantum-plane identity
\[
x_2^nx_1^m=q^{mn}x_1^mx_2^n.
\]
When $q=1$ and $r\neq 0$, however, one obtains
\[
\Lambda_{1,r}(n,m;k)
=
r^k\binom{n}{k}m^{\overline{k}}.
\]
Thus the Jordan relation $x_2x_1=x_1x_2+rx_1^2$ produces Lah--Whitney type coefficients. These are homogeneous relatives of the Stirling numbers appearing in classical normal ordering and fit naturally into the broader family of generalized Stirling and Whitney-type coefficients \cite{HsuShiue1998}. In particular, the Jordan cases in the Zhang--Zhang list should be treated as the first place where nontrivial Stirling-type combinatorics enters the theory.
\end{remark}

The next subsection records the family-wise internal parameters and the explicit mixed crossing data. This is the point where the general double Ore notation is specialized to the $26$ families in $\mathcal{LIST}$.

\subsection{Explicit family-wise mixed normal ordering coefficients}\label{SubsectionExplicitFamilywiseCoefficients}

We now make the mixed normal ordering coefficients completely explicit for the $26$ families in $\mathcal{LIST}$. The purpose of this subsection is to replace the abstract structural coefficients $a_{ijst}^{\bar{\mathbb{F}}}$ by concrete family-wise crossing data. This gives a uniform but explicit combinatorial description of the normal ordered form of every mixed word
\[
y_1^cy_2^dx_1^ax_2^b.
\]

We shall use the following compact notation. If a mixed relation contains a summand
\[
\lambda x_ty_\ell,
\]
we encode this summand as the triple
\[
(\lambda;t,\ell).
\]
Thus, for a fixed family $\bar{\mathbb{F}}$, we define finite sets
\[
\mathscr{K}_{i,s}^{\bar{\mathbb{F}}}
\subseteq
\Bbbk\times\{1,2\}\times\{1,2\},
\ i,s\in\{1,2\},
\]
by the rule
\begin{equation}\label{EquationCrossingKernelDefinition}
y_ix_s
=
\sum_{(\lambda;t,\ell)\in \mathscr{K}_{i,s}^{\bar{\mathbb{F}}}}
\lambda x_ty_\ell.
\end{equation}
In other words, $\mathscr{K}_{i,s}^{\bar{\mathbb{F}}}$ is the complete list of ordered summands appearing in the normal form of the forbidden pair $y_ix_s$.

\begin{definition}\label{DefinitionFamilyInternalParameters}
For each family $\bar{\mathbb{F}}$, let $(q_{\bar{\mathbb{F}}},r_{\bar{\mathbb{F}}})$ and $(p_{\bar{\mathbb{F}}},s_{\bar{\mathbb{F}}})$ be determined by
\[
x_2x_1=q_{\bar{\mathbb{F}}}x_1x_2+r_{\bar{\mathbb{F}}}x_1^2,
\
y_2y_1=p_{\bar{\mathbb{F}}}y_1y_2+s_{\bar{\mathbb{F}}}y_1^2.
\]
Thus:
\[
\begin{array}{c|c|c}
\bar{\mathbb{F}} & (q_{\bar{\mathbb{F}}},r_{\bar{\mathbb{F}}}) & (p_{\bar{\mathbb{F}}},s_{\bar{\mathbb{F}}})\\
\hline
\mathbb{A} & (1,0) & (1,1)\\
\mathbb{B},\mathbb{C} & (p,0) & (p,0)\\
\mathbb{D},\mathbb{E},\mathbb{F} & (-1,0) & (p,0)\\
\mathbb{G} & (1,0) & (p,0)\\
\mathbb{H} & (1,1) & (-1,0)\\
\mathbb{I},\mathbb{J},\mathbb{K},\mathbb{L} & (q,0) & (-1,0)\\
\mathbb{M},\mathbb{N},\mathbb{O},\mathbb{P},\mathbb{Q},\mathbb{R},\mathbb{S},\mathbb{T},\mathbb{U} & (-1,0) & (-1,0)\\
\mathbb{V},\mathbb{W},\mathbb{X},\mathbb{Y} & (1,0) & (-1,0)\\
\mathbb{Z} & (-1,0) & (1,0).
\end{array}
\]
\end{definition}

Hence
\begin{equation}\label{EquationInternalXFamily}
x_2^nx_1^m
=
\sum_{k=0}^{n}
\Lambda_{q_{\bar{\mathbb{F}}},r_{\bar{\mathbb{F}}}}(n,m;k)
x_1^{m+k}x_2^{n-k},
\end{equation}
and
\begin{equation}\label{EquationInternalYFamily}
y_2^ny_1^m
=
\sum_{k=0}^{n}
\Lambda_{p_{\bar{\mathbb{F}}},s_{\bar{\mathbb{F}}}}(n,m;k)
y_1^{m+k}y_2^{n-k}.
\end{equation}

\begin{remark}\label{RemarkJordanStirlingExplicit}
The coefficients in \eqref{EquationInternalXFamily} and \eqref{EquationInternalYFamily} specialize to scalar quantum factors whenever $r_{\bar{\mathbb{F}}}=0$ or $s_{\bar{\mathbb{F}}}=0$. The only Jordan contributions among the $26$ families are:
\[
\mathbb{A}:\ y_2y_1=y_1y_2+y_1^2,
\quad
\mathbb{H}:\ x_2x_1=x_1x_2+x_1^2.
\]
Consequently,
\[
\mathbb{A}:\
y_2^ny_1^m
=
\sum_{k=0}^{n}
\binom{n}{k}m^{\overline{k}}
y_1^{m+k}y_2^{n-k},
\]
and
\[
\mathbb{H}:\
x_2^nx_1^m
=
\sum_{k=0}^{n}
\binom{n}{k}m^{\overline{k}}
x_1^{m+k}x_2^{n-k}.
\]
Thus the families $\mathbb{A}$ and $\mathbb{H}$ carry the Lah--Whitney, or Stirling-type, layer of the normal ordering theory.
\end{remark}

We now list the crossing kernels $\mathscr{K}_{i,s}^{\bar{\mathbb{F}}}$ explicitly.

\begin{center}
\renewcommand{\arraystretch}{1.35}
\begin{longtable}{c|p{0.82\textwidth}}
\toprule
$\bar{\mathbb{F}}$ & Explicit mixed crossing kernels\\
\midrule
\endfirsthead

\toprule
$\bar{\mathbb{F}}$ & Explicit mixed crossing kernels\\
\midrule
\endhead

$\mathbb{A}$ &
$\mathscr{K}_{1,1}^{\mathbb{A}}=\{(1;1,1)\}$,

$\mathscr{K}_{1,2}^{\mathbb{A}}=\{(1;2,1),(1;1,2)\}$,

$\mathscr{K}_{2,1}^{\mathbb{A}}=\{(1;1,2)\}$,

$\mathscr{K}_{2,2}^{\mathbb{A}}=\{(-2;2,1),(-1;1,2),(1;2,2)\}$.
\\

$\mathbb{B}$ &
$\mathscr{K}_{1,1}^{\mathbb{B}}=\{(1;2,2)\}$,

$\mathscr{K}_{1,2}^{\mathbb{B}}=\{(1;1,2)\}$,

$\mathscr{K}_{2,1}^{\mathbb{B}}=\{(-1;2,1)\}$,

$\mathscr{K}_{2,2}^{\mathbb{B}}=\{(1;1,1)\}$.
\\

$\mathbb{C}$ &
$\mathscr{K}_{1,1}^{\mathbb{C}}=\{(-1;1,1),(p^2;2,1),(1;1,2),(-p;2,2)\}$,

$\mathscr{K}_{1,2}^{\mathbb{C}}=\{(-p;1,1),(1;2,1),(1;1,2),(-p;2,2)\}$,

$\mathscr{K}_{2,1}^{\mathbb{C}}=\{(-p;1,1),(-2p^2;2,1),(p;1,2),(-p;2,2)\}$,

$\mathscr{K}_{2,2}^{\mathbb{C}}=\{(-p;1,1),(p^2;2,1),(1;1,2),(-1;2,2)\}$.
\\

$\mathbb{D}$ &
$\mathscr{K}_{1,1}^{\mathbb{D}}=\{(-p;1,1)\}$,

$\mathscr{K}_{1,2}^{\mathbb{D}}=\{(-p^2;2,1),(1;1,2)\}$,

$\mathscr{K}_{2,1}^{\mathbb{D}}=\{(p;1,2)\}$,

$\mathscr{K}_{2,2}^{\mathbb{D}}=\{(1;1,1),(1;2,2)\}$.
\\

$\mathbb{E}$ &
$\mathscr{K}_{1,1}^{\mathbb{E}}=\{(1;1,2),(1;2,2)\}$,

$\mathscr{K}_{1,2}^{\mathbb{E}}=\{(1;1,2),(-1;2,2)\}$,

$\mathscr{K}_{2,1}^{\mathbb{E}}=\{(-1;1,1),(1;2,1)\}$,

$\mathscr{K}_{2,2}^{\mathbb{E}}=\{(1;1,1),(1;2,1)\}$.
\\

$\mathbb{F}$ &
$\mathscr{K}_{1,1}^{\mathbb{F}}=\{(-1;1,1),(-p;2,1),(1;1,2),(-1;2,2)\}$,

$\mathscr{K}_{1,2}^{\mathbb{F}}=\{(-p;1,1),(1;2,1),(1;1,2),(1;2,2)\}$,

$\mathscr{K}_{2,1}^{\mathbb{F}}=\{(-p;1,1),(p;2,1),(p;1,2),(1;2,2)\}$,

$\mathscr{K}_{2,2}^{\mathbb{F}}=\{(-p;1,1),(-p;2,1),(1;1,2),(-p;2,2)\}$.
\\

$\mathbb{G}$ &
$\mathscr{K}_{1,1}^{\mathbb{G}}=\{(p;1,1)\}$,

$\mathscr{K}_{1,2}^{\mathbb{G}}=\{(p;1,1),(p^2;2,1),(1;1,2)\}$,

$\mathscr{K}_{2,1}^{\mathbb{G}}=\{(p;1,2)\}$,

$\mathscr{K}_{2,2}^{\mathbb{G}}=\{(f;1,1),(-1;1,2),(1;2,2)\}$.
\\

$\mathbb{H}$ &
$\mathscr{K}_{1,1}^{\mathbb{H}}=\{(1;1,2)\}$,

$\mathscr{K}_{1,2}^{\mathbb{H}}=\{(f;1,2),(1;2,2)\}$,

$\mathscr{K}_{2,1}^{\mathbb{H}}=\{(1;1,1)\}$,

$\mathscr{K}_{2,2}^{\mathbb{H}}=\{(f;1,1),(1;2,1)\}$.
\\

$\mathbb{I}$ &
$\mathscr{K}_{1,1}^{\mathbb{I}}=\{(-q;1,1),(-q;2,1),(1;1,2),(-q;2,2)\}$,

$\mathscr{K}_{1,2}^{\mathbb{I}}=\{(1;1,1),(1;2,1),(1;1,2),(-q;2,2)\}$,

$\mathscr{K}_{2,1}^{\mathbb{I}}=\{(1;1,1),(q;2,1),(q;1,2),(-q;2,2)\}$,

$\mathscr{K}_{2,2}^{\mathbb{I}}=\{(-1;1,1),(-q;2,1),(1;1,2),(-1;2,2)\}$.
\\

$\mathbb{J}$ &
$\mathscr{K}_{1,1}^{\mathbb{J}}=\{(1;2,1),(1;2,2)\}$,

$\mathscr{K}_{1,2}^{\mathbb{J}}=\{(-1;1,1),(1;1,2)\}$,

$\mathscr{K}_{2,1}^{\mathbb{J}}=\{(1;2,1),(-1;2,2)\}$,

$\mathscr{K}_{2,2}^{\mathbb{J}}=\{(1;1,1),(1;1,2)\}$.
\\

$\mathbb{K}$ &
$\mathscr{K}_{1,1}^{\mathbb{K}}=\{(1;1,1)\}$,

$\mathscr{K}_{1,2}^{\mathbb{K}}=\{(1;2,2)\}$,

$\mathscr{K}_{2,1}^{\mathbb{K}}=\{(1;1,2)\}$,

$\mathscr{K}_{2,2}^{\mathbb{K}}=\{(f;2,1)\}$.
\\

$\mathbb{L}$ &
$\mathscr{K}_{1,1}^{\mathbb{L}}=\{(f;1,2)\}$,

$\mathscr{K}_{1,2}^{\mathbb{L}}=\{(1;2,2)\}$,

$\mathscr{K}_{2,1}^{\mathbb{L}}=\{(f;1,1)\}$,

$\mathscr{K}_{2,2}^{\mathbb{L}}=\{(1;2,1)\}$.
\\

$\mathbb{M}$ &
$\mathscr{K}_{1,1}^{\mathbb{M}}=\{(1;2,1),(1;1,2)\}$,

$\mathscr{K}_{1,2}^{\mathbb{M}}=\{(f;1,1),(-1;2,2)\}$,

$\mathscr{K}_{2,1}^{\mathbb{M}}=\{(1;1,1),(-1;2,2)\}$,

$\mathscr{K}_{2,2}^{\mathbb{M}}=\{(-1;2,1),(-f;1,2)\}$.
\\

$\mathbb{N}$ &
$\mathscr{K}_{1,1}^{\mathbb{N}}=\{(-g;2,1),(f;2,2)\}$,

$\mathscr{K}_{1,2}^{\mathbb{N}}=\{(g;1,1),(f;1,2)\}$,

$\mathscr{K}_{2,1}^{\mathbb{N}}=\{(f;2,1),(-g;2,2)\}$,

$\mathscr{K}_{2,2}^{\mathbb{N}}=\{(f;1,1),(g;1,2)\}$.
\\

$\mathbb{O}$ &
$\mathscr{K}_{1,1}^{\mathbb{O}}=\{(1;1,1),(f;2,2)\}$,

$\mathscr{K}_{1,2}^{\mathbb{O}}=\{(-1;2,1),(1;1,2)\}$,

$\mathscr{K}_{2,1}^{\mathbb{O}}=\{(f;2,1),(-1;1,2)\}$,

$\mathscr{K}_{2,2}^{\mathbb{O}}=\{(1;1,1),(1;2,2)\}$.
\\

$\mathbb{P}$ &
$\mathscr{K}_{1,1}^{\mathbb{P}}=\{(1;1,2),(f;2,2)\}$,

$\mathscr{K}_{1,2}^{\mathbb{P}}=\{(1;1,2),(1;2,2)\}$,

$\mathscr{K}_{2,1}^{\mathbb{P}}=\{(1;1,1),(-f;2,1)\}$,

$\mathscr{K}_{2,2}^{\mathbb{P}}=\{(-1;1,1),(1;2,1)\}$.
\\

$\mathbb{Q}$ &
$\mathscr{K}_{1,1}^{\mathbb{Q}}=\{(1;1,2)\}$,

$\mathscr{K}_{1,2}^{\mathbb{Q}}=\{(1;1,1),(1;2,1),(1;1,2)\}$,

$\mathscr{K}_{2,1}^{\mathbb{Q}}=\{(-1;1,1)\}$,

$\mathscr{K}_{2,2}^{\mathbb{Q}}=\{(1;1,1),(-1;1,2),(1;2,2)\}$.
\\

$\mathbb{R}$ &
$\mathscr{K}_{1,1}^{\mathbb{R}}=\{(1;1,1),(1;2,1),(1;1,2)\}$,

$\mathscr{K}_{1,2}^{\mathbb{R}}=\{(1;1,2)\}$,

$\mathscr{K}_{2,1}^{\mathbb{R}}=\{(1;2,1)\}$,

$\mathscr{K}_{2,2}^{\mathbb{R}}=\{(-1;2,1),(-1;1,2),(1;2,2)\}$.
\\

$\mathbb{S}$ &
$\mathscr{K}_{1,1}^{\mathbb{S}}=\{(-1;1,1),(1;2,1),(1;1,2),(1;2,2)\}$,

$\mathscr{K}_{1,2}^{\mathbb{S}}=\{(1;1,1),(-1;2,1),(1;1,2),(1;2,2)\}$,

$\mathscr{K}_{2,1}^{\mathbb{S}}=\{(1;1,1),(1;2,1),(-1;1,2),(1;2,2)\}$,

$\mathscr{K}_{2,2}^{\mathbb{S}}=\{(1;1,1),(1;2,1),(1;1,2),(-1;2,2)\}$.
\\

$\mathbb{T}$ &
$\mathscr{K}_{1,1}^{\mathbb{T}}=\{(-1;1,1),(1;2,1),(1;1,2),(1;2,2)\}$,

$\mathscr{K}_{1,2}^{\mathbb{T}}=\{(1;1,1),(-1;2,1),(1;1,2),(1;2,2)\}$,

$\mathscr{K}_{2,1}^{\mathbb{T}}=\{(1;1,1),(1;2,1),(1;1,2),(-1;2,2)\}$,

$\mathscr{K}_{2,2}^{\mathbb{T}}=\{(1;1,1),(1;2,1),(-1;1,2),(1;2,2)\}$.
\\

$\mathbb{U}$ &
$\mathscr{K}_{1,1}^{\mathbb{U}}=\{(-1;1,1),(1;2,1),(1;1,2),(1;2,2)\}$,

$\mathscr{K}_{1,2}^{\mathbb{U}}=\{(1;1,1),(1;2,1),(1;1,2),(-1;2,2)\}$,

$\mathscr{K}_{2,1}^{\mathbb{U}}=\{(1;1,1),(1;2,1),(-1;1,2),(1;2,2)\}$,

$\mathscr{K}_{2,2}^{\mathbb{U}}=\{(1;1,1),(-1;2,1),(1;1,2),(1;2,2)\}$.
\\

$\mathbb{V}$ &
$\mathscr{K}_{1,1}^{\mathbb{V}}=\{(1;2,1),(1;1,2)\}$,

$\mathscr{K}_{1,2}^{\mathbb{V}}=\{(1;2,1)\}$,

$\mathscr{K}_{2,1}^{\mathbb{V}}=\{(-1;1,1),(1;2,1)\}$,

$\mathscr{K}_{2,2}^{\mathbb{V}}=\{(1;2,2)\}$.
\\

$\mathbb{W}$ &
$\mathscr{K}_{1,1}^{\mathbb{W}}=\{(f;2,1),(1;1,2)\}$,

$\mathscr{K}_{1,2}^{\mathbb{W}}=\{(1;1,1),(-1;2,2)\}$,

$\mathscr{K}_{2,1}^{\mathbb{W}}=\{(1;1,1),(f;2,2)\}$,

$\mathscr{K}_{2,2}^{\mathbb{W}}=\{(-1;2,1),(1;1,2)\}$.
\\

$\mathbb{X}$ &
$\mathscr{K}_{1,1}^{\mathbb{X}}=\{(1;1,2)\}$,

$\mathscr{K}_{1,2}^{\mathbb{X}}=\{(1;1,2),(1;2,2)\}$,

$\mathscr{K}_{2,1}^{\mathbb{X}}=\{(1;1,1)\}$,

$\mathscr{K}_{2,2}^{\mathbb{X}}=\{(1;1,1),(1;2,1)\}$.
\\

$\mathbb{Y}$ &
$\mathscr{K}_{1,1}^{\mathbb{Y}}=\{(1;1,1)\}$,

$\mathscr{K}_{1,2}^{\mathbb{Y}}=\{(f;1,1),(-1;2,1),(1;1,2)\}$,

$\mathscr{K}_{2,1}^{\mathbb{Y}}=\{(1;1,2)\}$,

$\mathscr{K}_{2,2}^{\mathbb{Y}}=\{(1;1,1),(f;1,2),(-1;2,2)\}$.
\\

$\mathbb{Z}$ &
$\mathscr{K}_{1,1}^{\mathbb{Z}}=\{(1;1,1),(1;2,2)\}$,

$\mathscr{K}_{1,2}^{\mathbb{Z}}=\{(1;2,1),(1;1,2)\}$,

$\mathscr{K}_{2,1}^{\mathbb{Z}}=\{(f;2,1),(-1;1,2)\}$,

$\mathscr{K}_{2,2}^{\mathbb{Z}}=\{(f;1,1),(-1;2,2)\}$.
\\

\bottomrule
\end{longtable}
\end{center}

We next define the explicit coefficient arrays which encode the normal ordering of one $y$-letter through an arbitrary ordered $x$-block.

\begin{definition}\label{DefinitionOmegaPsiFamily}
For a fixed family $\bar{\mathbb{F}}$, define
\[
\Omega_{\bar{\mathbb{F}}}(\mu,\nu;t;\alpha,\beta)
\]
by
\[
{\rm NF}_{\bar{\mathbb{F}}}(x_1^\mu x_2^\nu x_t)
=
\sum_{\alpha,\beta\ge0}
\Omega_{\bar{\mathbb{F}}}(\mu,\nu;t;\alpha,\beta)
x_1^\alpha x_2^\beta.
\]
Explicitly,
\[
\Omega_{\bar{\mathbb{F}}}(\mu,\nu;2;\alpha,\beta)
=
\delta_{\alpha,\mu}\delta_{\beta,\nu+1},
\]
and
\[
\Omega_{\bar{\mathbb{F}}}(\mu,\nu;1;\alpha,\beta)
=
\sum_{k=0}^{\nu}
\Lambda_{q_{\bar{\mathbb{F}}},r_{\bar{\mathbb{F}}}}(\nu,1;k)
\delta_{\alpha,\mu+1+k}
\delta_{\beta,\nu-k}.
\]

Similarly, define
\[
\Psi_{\bar{\mathbb{F}}}(u,v;\ell;\gamma,\delta)
\]
by
\[
{\rm NF}_{\bar{\mathbb{F}}}(y_1^uy_2^vy_\ell)
=
\sum_{\gamma,\delta\ge0}
\Psi_{\bar{\mathbb{F}}}(u,v;\ell;\gamma,\delta)
y_1^\gamma y_2^\delta.
\]
Explicitly,
\[
\Psi_{\bar{\mathbb{F}}}(u,v;2;\gamma,\delta)
=
\delta_{\gamma,u}\delta_{\delta,v+1},
\]
and
\[
\Psi_{\bar{\mathbb{F}}}(u,v;1;\gamma,\delta)
=
\sum_{k=0}^{v}
\Lambda_{p_{\bar{\mathbb{F}}},s_{\bar{\mathbb{F}}}}(v,1;k)
\delta_{\gamma,u+1+k}
\delta_{\delta,v-k}.
\]
\end{definition}

\begin{proposition}\label{PropositionSingleMixedExplicit}
Fix a family $\bar{\mathbb{F}}\in\{\mathbb{A},\ldots,\mathbb{Z}\}$. For $i\in\{1,2\}$ and $a,b\in\bar{\mathbb{N}}$, there exist unique coefficients
\[
W_{\bar{\mathbb{F}},i}^{(a,b)}(\alpha,\beta;\ell)\in\Bbbk,
\
\alpha,\beta\in\bar{\mathbb{N}},\ \ell\in\{1,2\},
\]
such that
\begin{equation}\label{EquationSingleMixedExplicitExpansion}
{\rm NF}_{\bar{\mathbb{F}}}(y_ix_1^ax_2^b)
=
\sum_{\alpha,\beta\ge0}
\sum_{\ell=1}^{2}
W_{\bar{\mathbb{F}},i}^{(a,b)}(\alpha,\beta;\ell)
x_1^\alpha x_2^\beta y_\ell.
\end{equation}
They are determined by the initial conditions
\[
W_{\bar{\mathbb{F}},i}^{(0,0)}(0,0;\ell)=\delta_{i,\ell},
\
W_{\bar{\mathbb{F}},i}^{(0,0)}(\alpha,\beta;\ell)=0
\text{ if }(\alpha,\beta)\neq(0,0),
\]
and by the following explicit recursion.

If $b>0$, set
\[
(a',b',\varepsilon)=(a,b-1,2).
\]
If $b=0$ and $a>0$, set
\[
(a',b',\varepsilon)=(a-1,0,1).
\]
Then, for $(a,b)\neq(0,0)$,
\begin{align}
&W_{\bar{\mathbb{F}},i}^{(a,b)}(\alpha,\beta;\ell)
\notag\\
&=
\sum_{\mu,\nu\ge0}
\sum_{j=1}^{2}
W_{\bar{\mathbb{F}},i}^{(a',b')}(\mu,\nu;j)
\sum_{(\lambda;t,\ell)\in\mathscr{K}_{j,\varepsilon}^{\bar{\mathbb{F}}}}
\lambda\,
\Omega_{\bar{\mathbb{F}}}(\mu,\nu;t;\alpha,\beta).
\label{EquationSingleMixedExplicitRecursion}
\end{align}
\end{proposition}

\begin{proof}
The proof is by induction on $a+b$. The case $a=b=0$ is immediate. Assume $(a,b)\neq(0,0)$. If $b>0$, then
\[
y_ix_1^ax_2^b=(y_ix_1^ax_2^{b-1})x_2.
\]
If $b=0$ and $a>0$, then
\[
y_ix_1^a=(y_ix_1^{a-1})x_1.
\]
In either case, the induction hypothesis gives the normal ordered expansion of the first factor. Multiplying by $x_\varepsilon$ on the right creates subwords $y_jx_\varepsilon$, and these are replaced explicitly using the family-wise kernel $\mathscr{K}_{j,\varepsilon}^{\bar{\mathbb{F}}}$. Finally, the possible internal disorder in the $x$-part is normalized by the coefficients $\Omega_{\bar{\mathbb{F}}}$. Collecting the coefficient of $x_1^\alpha x_2^\beta y_\ell$ gives \eqref{EquationSingleMixedExplicitRecursion}. Uniqueness follows from the PBW basis.
\end{proof}

The coefficients $W_{\bar{\mathbb{F}},i}^{(a,b)}(\alpha,\beta;\ell)$ are the double Ore analogues of the two-letter normal ordering coefficients in the three-dimensional case. They are the first level of the combinatorics: they record all weighted paths by which a single $y_i$ crosses the ordered block $x_1^ax_2^b$.

We now pass from one $y$-letter to arbitrary $y$-blocks.

\begin{proposition}\label{PropositionMixedBlockExplicit}
Fix a family $\bar{\mathbb{F}}\in\{\mathbb{A},\ldots,\mathbb{Z}\}$. For $a,b,c,d\in\bar{\mathbb{N}}$, there exist unique coefficients
\[
\Theta_{\bar{\mathbb{F}}}(a,b;c,d;\alpha,\beta;\gamma,\delta)\in\Bbbk
\]
such that
\begin{equation}\label{EquationThetaExplicitExpansion}
{\rm NF}_{\bar{\mathbb{F}}}(y_1^cy_2^dx_1^ax_2^b)
=
\sum_{\alpha,\beta,\gamma,\delta\ge0}
\Theta_{\bar{\mathbb{F}}}(a,b;c,d;\alpha,\beta;\gamma,\delta)
x_1^\alpha x_2^\beta y_1^\gamma y_2^\delta.
\end{equation}
The initial conditions are
\[
\Theta_{\bar{\mathbb{F}}}(a,b;0,0;\alpha,\beta;0,0)
=
\delta_{\alpha,a}\delta_{\beta,b},
\]
and
\[
\Theta_{\bar{\mathbb{F}}}(a,b;0,0;\alpha,\beta;\gamma,\delta)=0
\text{ if }(\gamma,\delta)\neq(0,0).
\]

If $c+d>0$, define
\[
(c',d',i_0)=
\begin{cases}
(c,d-1,2), & d>0,\\
(c-1,0,1), & d=0,\ c>0.
\end{cases}
\]
Then
\begin{align}
&\Theta_{\bar{\mathbb{F}}}(a,b;c,d;\alpha,\beta;\gamma,\delta)
\notag\\
&=
\sum_{\mu,\nu\ge0}
\sum_{\ell=1}^{2}
\sum_{u,v\ge0}
W_{\bar{\mathbb{F}},i_0}^{(a,b)}(\mu,\nu;\ell)
\Theta_{\bar{\mathbb{F}}}(\mu,\nu;c',d';\alpha,\beta;u,v)
\Psi_{\bar{\mathbb{F}}}(u,v;\ell;\gamma,\delta).
\label{EquationThetaExplicitRecursion}
\end{align}
\end{proposition}

\begin{proof}
If $c=d=0$, then the word is already an ordered $x$-block. Suppose $c+d>0$. If $d>0$, write
\[
y_1^cy_2^dx_1^ax_2^b
=
y_1^cy_2^{d-1}(y_2x_1^ax_2^b).
\]
If $d=0$ and $c>0$, write
\[
y_1^cx_1^ax_2^b
=
y_1^{c-1}(y_1x_1^ax_2^b).
\]
The inner factor is normalized by Proposition \ref{PropositionSingleMixedExplicit}. Then the remaining $y$-block is moved through the resulting $x$-block by the induction hypothesis. Finally, right multiplication by the remaining letter $y_\ell$ is normalized by $\Psi_{\bar{\mathbb{F}}}$. This gives \eqref{EquationThetaExplicitRecursion}. PBW uniqueness gives uniqueness of the coefficients.
\end{proof}

\begin{remark}\label{RemarkThetaWeightedPaths}
The coefficient
\[
\Theta_{\bar{\mathbb{F}}}(a,b;c,d;\alpha,\beta;\gamma,\delta)
\]
is a weighted sum over all admissible reduction paths sending the forbidden word
\[
y_1^cy_2^dx_1^ax_2^b
\]
to the normal monomial
\[
x_1^\alpha x_2^\beta y_1^\gamma y_2^\delta.
\]
Every local crossing contributes one of the explicit weights appearing in the kernels $\mathscr{K}_{i,s}^{\bar{\mathbb{F}}}$, while every internal correction in the $x$- or $y$-part contributes a coefficient $\Lambda_{q_{\bar{\mathbb{F}}},r_{\bar{\mathbb{F}}}}$ or $\Lambda_{p_{\bar{\mathbb{F}}},s_{\bar{\mathbb{F}}}}$. Thus, in the Jordan families $\mathbb{A}$ and $\mathbb{H}$, the coefficients $\Theta_{\bar{\mathbb{F}}}$ contain Lah--Whitney, hence Stirling-type, substructures.
\end{remark}

We now give the explicit structure constants for products of arbitrary normal monomials.

\begin{proposition}\label{PropositionExplicitProductConstants}
Fix a family $\bar{\mathbb{F}}\in\{\mathbb{A},\ldots,\mathbb{Z}\}$. For normal monomials
\[
M=x_1^ax_2^by_1^cy_2^d,
\qquad
N=x_1^{a'}x_2^{b'}y_1^{c'}y_2^{d'},
\]
one has
\begin{align}
&{\rm NF}_{\bar{\mathbb{F}}}(MN)
\notag\\
&=
\sum_{A,B,C,D\ge0}
\mathfrak{m}_{\bar{\mathbb{F}}}
\begin{bmatrix}
a,b;c,d\\
a',b';c',d'\\
A,B;C,D
\end{bmatrix}
x_1^Ax_2^By_1^Cy_2^D,
\label{EquationProductConstantExpansionExplicit}
\end{align}
where the product constants are given explicitly by
\begin{align}
&\mathfrak{m}_{\bar{\mathbb{F}}}
\begin{bmatrix}
a,b;c,d\\
a',b';c',d'\\
A,B;C,D
\end{bmatrix}
\notag\\
&=
\sum_{\alpha,\beta,\gamma,\delta\ge0}
\sum_{r,s\ge0}
\Theta_{\bar{\mathbb{F}}}(a',b';c,d;\alpha,\beta;\gamma,\delta)
\Lambda_{q_{\bar{\mathbb{F}}},r_{\bar{\mathbb{F}}}}(b,\alpha;r)
\Lambda_{p_{\bar{\mathbb{F}}},s_{\bar{\mathbb{F}}}}(\delta,c';s)
\notag\\
&\qquad\qquad\qquad\qquad
\times
\delta_{A,a+\alpha+r}
\delta_{B,b-r+\beta}
\delta_{C,\gamma+c'+s}
\delta_{D,\delta-s+d'}.
\label{EquationProductConstantsExplicitFormula}
\end{align}
\end{proposition}

\begin{proof}
We have
\[
MN
=
x_1^ax_2^b
\left(
y_1^cy_2^dx_1^{a'}x_2^{b'}
\right)
y_1^{c'}y_2^{d'}.
\]
The middle factor is normalized by Proposition \ref{PropositionMixedBlockExplicit}. Thus it is enough to normalize
\[
x_1^ax_2^b x_1^\alpha x_2^\beta
\]
and
\[
y_1^\gamma y_2^\delta y_1^{c'}y_2^{d'}.
\]
The first one is governed by \eqref{EquationInternalXFamily}, namely
\[
x_1^ax_2^b x_1^\alpha x_2^\beta
=
\sum_{r\ge0}
\Lambda_{q_{\bar{\mathbb{F}}},r_{\bar{\mathbb{F}}}}(b,\alpha;r)
x_1^{a+\alpha+r}x_2^{b-r+\beta}.
\]
The second one is governed by \eqref{EquationInternalYFamily}, namely
\[
y_1^\gamma y_2^\delta y_1^{c'}y_2^{d'}
=
\sum_{s\ge0}
\Lambda_{p_{\bar{\mathbb{F}}},s_{\bar{\mathbb{F}}}}(\delta,c';s)
y_1^{\gamma+c'+s}y_2^{\delta-s+d'}.
\]
Collecting coefficients gives \eqref{EquationProductConstantsExplicitFormula}.
\end{proof}

The preceding proposition gives the explicit multiplication table of the PBW basis. Therefore all powers of normal blocks are obtained by a direct recurrence.

\begin{proposition}\label{PropositionExplicitBlockPowersFamily}
Fix a family $\bar{\mathbb{F}}\in\{\mathbb{A},\ldots,\mathbb{Z}\}$ and let
\[
M_{a,b,c,d}:=x_1^ax_2^by_1^cy_2^d.
\]
For every $N\ge1$, there exist unique coefficients
\[
R_{\bar{\mathbb{F}},N}^{(a,b,c,d)}(A,B,C,D)\in\Bbbk
\]
such that
\begin{equation}\label{EquationExplicitBlockPowerFamily}
M_{a,b,c,d}^{N}
=
\sum_{A,B,C,D\ge0}
R_{\bar{\mathbb{F}},N}^{(a,b,c,d)}(A,B,C,D)
x_1^Ax_2^By_1^Cy_2^D.
\end{equation}
They satisfy
\[
R_{\bar{\mathbb{F}},1}^{(a,b,c,d)}(A,B,C,D)
=
\delta_{A,a}\delta_{B,b}\delta_{C,c}\delta_{D,d},
\]
and, for $N\ge1$,
\begin{align}
&R_{\bar{\mathbb{F}},N+1}^{(a,b,c,d)}(A,B,C,D)
\notag\\
&=
\sum_{A',B',C',D'\ge0}
R_{\bar{\mathbb{F}},N}^{(a,b,c,d)}(A',B',C',D')
\mathfrak{m}_{\bar{\mathbb{F}}}
\begin{bmatrix}
A',B';C',D'\\
a,b;c,d\\
A,B;C,D
\end{bmatrix}.
\label{EquationExplicitBlockPowerRecursionFamily}
\end{align}
\end{proposition}

\begin{proof}
This follows by multiplying the PBW expansion of $M_{a,b,c,d}^{N}$ by $M_{a,b,c,d}$ on the right and applying Proposition \ref{PropositionExplicitProductConstants}.
\end{proof}

\begin{corollary}\label{CorollaryFullNormalOrderingFamily}
For each family $\bar{\mathbb{F}}\in\{\mathbb{A},\ldots,\mathbb{Z}\}$, the following normal ordering problems are completely determined by the explicit coefficient arrays above:
\begin{enumerate}
    \item[\rm (1)] the internal words $x_2^nx_1^m$ and $y_2^ny_1^m$, via \eqref{EquationInternalXFamily} and \eqref{EquationInternalYFamily};
    \item[\rm (2)] the mixed one-letter words $y_ix_1^ax_2^b$, via \eqref{EquationSingleMixedExplicitExpansion} and \eqref{EquationSingleMixedExplicitRecursion};
    \item[\rm (3)] the mixed block words $y_1^cy_2^dx_1^ax_2^b$, via \eqref{EquationThetaExplicitExpansion} and \eqref{EquationThetaExplicitRecursion};
    \item[\rm (4)] the product of arbitrary normal monomials, via \eqref{EquationProductConstantExpansionExplicit} and \eqref{EquationProductConstantsExplicitFormula};
    \item[\rm (5)] the powers of normal blocks $(x_1^ax_2^by_1^cy_2^d)^N$, via \eqref{EquationExplicitBlockPowerFamily} and \eqref{EquationExplicitBlockPowerRecursionFamily}.
\end{enumerate}
\end{corollary}

\begin{remark}\label{RemarkHowToReadFamilies}
For a concrete family, no abstract coefficient $a_{ijst}$ remains. One reads the four explicit sets
\[
\mathscr{K}_{1,1}^{\bar{\mathbb{F}}},\
\mathscr{K}_{1,2}^{\bar{\mathbb{F}}},\
\mathscr{K}_{2,1}^{\bar{\mathbb{F}}},\
\mathscr{K}_{2,2}^{\bar{\mathbb{F}}}
\]
from the family-wise table above and inserts them into \eqref{EquationSingleMixedExplicitRecursion}. This gives the same type of explicit coefficient recursion as the arrays $W,U,V,R$ used in the normal ordering theory of three-dimensional skew polynomial rings, but now in the four-generator double Ore setting.
\end{remark}

Finally, we record a completely explicit noncommutative multinomial recurrence. It will be useful whenever one wants to compute expressions such as $(x_1+x_2)^N$, $(y_1+y_2)^N$, $(x_i+y_j)^N$ or $(x_1+x_2+y_1+y_2)^N$.

\begin{proposition}\label{PropositionExplicitPolynomialPowers}
Fix a family $\bar{\mathbb{F}}\in\{\mathbb{A},\ldots,\mathbb{Z}\}$. Let
\[
L=
\lambda_1x_1+\lambda_2x_2+\mu_1y_1+\mu_2y_2,
\qquad
\lambda_1,\lambda_2,\mu_1,\mu_2\in\Bbbk.
\]
For every $N\ge0$, there exist unique coefficients
\[
E_{\bar{\mathbb{F}},N}^{L}(A,B,C,D)\in\Bbbk
\]
such that
\[
L^N
=
\sum_{A,B,C,D\ge0}
E_{\bar{\mathbb{F}},N}^{L}(A,B,C,D)
x_1^Ax_2^By_1^Cy_2^D.
\]
They are determined by
\[
E_{\bar{\mathbb{F}},0}^{L}(0,0,0,0)=1,
\qquad
E_{\bar{\mathbb{F}},0}^{L}(A,B,C,D)=0
\quad\text{if }(A,B,C,D)\neq(0,0,0,0),
\]
and by the recurrence
\begin{align}
&E_{\bar{\mathbb{F}},N+1}^{L}(A,B,C,D)
\notag\\
&=
\lambda_1
\sum_{A',B',C',D'\ge0}
E_{\bar{\mathbb{F}},N}^{L}(A',B',C',D')
\mathfrak{m}_{\bar{\mathbb{F}}}
\begin{bmatrix}
A',B';C',D'\\
1,0;0,0\\
A,B;C,D
\end{bmatrix}
\notag\\
&\quad+
\lambda_2
\sum_{A',B',C',D'\ge0}
E_{\bar{\mathbb{F}},N}^{L}(A',B',C',D')
\mathfrak{m}_{\bar{\mathbb{F}}}
\begin{bmatrix}
A',B';C',D'\\
0,1;0,0\\
A,B;C,D
\end{bmatrix}
\notag\\
&\quad+
\mu_1
\sum_{A',B',C',D'\ge0}
E_{\bar{\mathbb{F}},N}^{L}(A',B',C',D')
\mathfrak{m}_{\bar{\mathbb{F}}}
\begin{bmatrix}
A',B';C',D'\\
0,0;1,0\\
A,B;C,D
\end{bmatrix}
\notag\\
&\quad+
\mu_2
\sum_{A',B',C',D'\ge0}
E_{\bar{\mathbb{F}},N}^{L}(A',B',C',D')
\mathfrak{m}_{\bar{\mathbb{F}}}
\begin{bmatrix}
A',B';C',D'\\
0,0;0,1\\
A,B;C,D
\end{bmatrix}.
\label{EquationExplicitLinearPowerRecursion}
\end{align}
\end{proposition}

\begin{proof}
The claim follows by expanding
\[
L^{N+1}=L^N(\lambda_1x_1+\lambda_2x_2+\mu_1y_1+\mu_2y_2)
\]
and applying Proposition \ref{PropositionExplicitProductConstants} to each of the four right multiplications.
\end{proof}

\begin{remark}\label{RemarkComputationalImplementationNormalOrdering}
The recursions above are finite and completely explicit. Nevertheless, for families such as $\mathbb{C},\mathbb{F},\mathbb{I},\mathbb{S},\mathbb{T}$ and $\mathbb{U}$, the expansions grow very quickly because all four mixed relations contain several summands. In those cases, the formulas should be understood as computable coefficient recursions rather than as expressions meant to be expanded by hand. This is precisely the situation where the \textsc{SageMath} implementation is useful: it evaluates the same recursive PBW reductions and verifies the resulting normal forms.
\end{remark}

\section{Computational implementation}\label{SubsectionComputationalImplementation}

Some of the PBW reductions appearing in this paper were verified with the aid of \textsc{SageMath} \cite{SageMath2026}. The implementation uses the fixed order
$x_1\prec x_2\prec y_1\prec y_2$ and represents noncommutative polynomials as finite dictionaries whose keys are words in the alphabet $\{x_1,x_2,y_1,y_2\}$ and whose values are coefficients in the corresponding parameter field. The rewriting rules are precisely those induced by the defining relations
$$
x_2x_1=q_{12}x_1x_2+q_{11}x_1^2,\
y_2y_1=p_{12}y_1y_2+p_{11}y_1^2,
$$
together with the four mixed relations
$$
y_ix_s=\sum_{j=1}^2\sum_{t=1}^2 a_{ijst}x_ty_j,
\ i,s\in\{1,2\}.
$$
The code iteratively replaces forbidden adjacent pairs by their ordered expressions and computes PBW normal forms, products, powers, commutators, overlap checks, and centrality systems in fixed bidegrees.

\section{Future work}\label{Futurework}

The recursive normal ordering theory developed here opens several natural directions. First, it would be interesting to extract closed formulas for the mixed coefficients
\[
W_{\bar{\mathbb{F}},i}^{(a,b)}(\alpha,\beta;\ell)
\text{ and }
\Theta_{\bar{\mathbb{F}}}(a,b;c,d;\alpha,\beta;\gamma,\delta)
\]
in the sparsest families, for instance $\mathbb{A},\mathbb{B},\mathbb{D},\mathbb{H},\mathbb{K},\mathbb{L},\mathbb{X}$ and $\mathbb{Z}$. In these cases the crossing kernels have few summands, so one may expect more transparent lattice-path models, generating functions, and perhaps closed triangular arrays generalizing the Lah--Whitney coefficients found in the internal Jordan relations.

A second direction is the systematic study of roots of unity. When the parameters $p$ or $q$ are roots of unity, the quantum factors in the internal normal ordering formulas may force cancellations and periodicities in the coefficient arrays. Such phenomena are closely related to central elements and to finite generation over central subalgebras. Therefore, the present normal ordering formulas could be used as a combinatorial tool for future computations of centers, discriminants and automorphism invariants of double extension regular algebras of type $(14641)$.

A third direction is to extend the method to non-trimmed double Ore extensions. The presence of derivation terms and tails introduces lower-degree contributions in the rewriting rules, and one should then expect inhomogeneous normal ordering coefficients closer to the classical Stirling numbers of the Weyl algebra. This would provide a bridge between the homogeneous Lah--Whitney arrays appearing in the present paper and the ordinary Stirling-type coefficients coming from affine commutation relations.

Finally, the algorithms used here can be developed into a broader computational package for PBW normal ordering in double extensions and related skew PBW extensions. Such a package would make it possible to search automatically for closed forms, test conjectural recurrences, compute commutators in prescribed bidegrees, and explore normal ordering phenomena in other regular algebras with structured rewriting systems.
\section*{Code availability}

The \textsc{SageMath} code used to support the symbolic PBW computations in this work is included as an ancillary file, \texttt{double\_ore\_pbw.sage}, accompanying the arXiv version of this manuscript.

\end{document}